\documentclass{amsart}
\usepackage{geometry}
\usepackage{amsmath,amsfonts,amsthm,amssymb,epsfig}
\usepackage{verbatim}
\usepackage[all]{xy}
\usepackage{enumerate}
\usepackage{hyperref}
\newtheorem{theorem}{Theorem}[section]
\newtheorem{lemma}[theorem]{Lemma}
\newtheorem{proposition}[theorem]{Proposition}
\newtheorem{corollary}[theorem]{Corollary}

\theoremstyle{definition}
\newtheorem{definition}[theorem]{Definition}

\newtheorem{remark}[theorem]{Remark}

\newcommand\bcdot{\ensuremath{%
  \mathchoice%
   {\mskip\thinmuskip\lower0.2ex\hbox{\scalebox{1.5}{$\cdot$}}\mskip\thinmuskip}}%
   {\mskip\thinmuskip\lower0.2ex\hbox{\scalebox{1.5}{$\cdot$}}\mskip\thinmuskip}%
   {\lower0.3ex\hbox{\scalebox{1.2}{$\cdot$}}}%
   {\lower0.3ex\hbox{\scalebox{1.2}{$\cdot$}}}%
   }

\renewcommand{\sp}{\mathrm{Sp}}
\newcommand{\fun}{\mathrm{Fun}}
\begin{document}

\title{A short proof of telescopic Tate vanishing}

\author{Dustin Clausen}
\email{dustin.clausen@math.ku.dk}
\address{Department of Mathematics, University of Copenhagen, Copenhagen, Denmark}

\author{Akhil Mathew}
\email{amathew@math.harvard.edu}
\address{Department of Mathematics, Harvard University, One Oxford Street, Cambridge, MA 02138}

\maketitle

\begin{abstract}
We give a short proof of a theorem of Kuhn that Tate constructions for finite
group actions vanish in telescopically localized stable homotopy theory.
In particular, we observe that Kuhn's theorem is equivalent to the statement that the transfer 
$BC_{p+} \to S^0$ admits a section after telescopic localization, which in turn
follows from the Kahn-Priddy theorem. 
\end{abstract}

\section{Introduction}

Let $\sp$ denote the $\infty$-category of spectra.\footnote{The use of
$\infty$-categories is not really necessary here. We use it for convenience in
discussing group actions. The reader can replace $\fun(BG, \sp)$ (of which we
only need the homotopy category) with the
subcategory of the homotopy category of genuine $G$-spectra given by the
Borel-equivariant or cofree ones. Compare the discussion in \cite[Sec.
6.3]{MNN17}.}  
Thanks to the thick subcategory theorem \cite{Hos98}, the ``primes" of $\sp$ 
(in the sense of 
\cite{Balmerspec}) 
are indexed
by the Morava $K$-theories $K(n)$, for $n\geq 0$ and an implicit prime $p$.  In
chromatic homotopy theory, one studies $\sp$ by first studying the
Bousfield-localized categories $L_{K(n)}\sp$, and then attempting to assemble the local knowledge into global knowledge.

One reason this is a viable approach is that the $\infty$-categories
$L_{K(n)}\sp$ have some surprisingly simple properties.  An important example is
the following theorem, which follows from the main results of
\cite{GrS96, HoS96} and has recently been extended in an interesting direction in
\cite{Ambidexterity}. 

\begin{theorem}[{Compare Greenlees-Sadofsky \cite{GrS96}}, Hovey-Sadofsky
\cite{HoS96}] \label{Kn}
Let $G$ be a finite group, and $X$ be a $K(n)$-local spectrum with a
$G$-action, i.e., an object of the $\infty$-category $\fun(BG, L_{K(n)} \sp)$.  Then the \emph{norm map}
$$X_{hG}\rightarrow X^{hG}$$
in $L_{K(n)}\sp$ is an equivalence.
\end{theorem}

When $n=0$, applying the localization $L_{K(0)}$ is equivalent to working
rationally.  Then this theorem is easy to prove, because the obvious
composition $X^{hG}\rightarrow X\rightarrow X_{hG}$ provides an inverse to the
norm map up to multiplication by the order of $G$.  But for $n>0$, Theorem
\ref{Kn} is surprising, since for example $X$ can easily have torsion dividing
the order of $G$.  The proof of \cite{GrS96, HoS96} is based on the calculation of $K(n)^\ast(
BG)$ for $G$ of prime order.

In addition to the Morava $K$-theory localization functors $L_{K(n)}$, there
are also the closely related telescopic localization functors $L_{T(n)}$.  There
is a natural transformation $L_{T(n)}\rightarrow L_{K(n)}L_{T(n)}=L_{K(n)}$
which is an equivalence for $n=0$ and $n=1$ (proved by Miller \cite{Mil81} at
odd primes and Mahowald \cite{Ma82} at $p=2$), but is believed not to be an
equivalence in general; this question is the \emph{telescope conjecture}.  In contrast to the $L_{K(n)}$, these $L_{T(n)}$ have a more fundamental finitary construction, but for $n>1$ it is not known how to compute their values explicitly.

In \cite[Th. 1.5]{Kuh04a}, Kuhn strengthened Theorem \ref{Kn} to apply to the telescopic localization functors.  That is, he showed that Theorem \ref{Kn} holds with $L_{T(n)}\sp$ replacing $L_{K(n)}\sp$.
Part of this had been previously proved by Mahowald-Shick \cite{MaS} at the
prime $2$. 

Without calculational access to the $L_{T(n)}$, Kuhn's proof is necessarily
different from that of \cite{GrS96}.  Instead it is based on  the
\emph{Bousfield-Kuhn functor}.  For $n  > 0$, this is a
functor 
$\Phi\colon\mathcal{S}_\ast\rightarrow L_{T(n)}\sp$ from pointed spaces to $T(n)$-local spectra such that
we have a natural equivalence
$$\Phi\circ\Omega^\infty \simeq L_{T(n)} \colon \sp \to L_{T(n)}\sp.$$
We refer to \cite{Kuh89, Bou01} for the construction of the functor and
\cite{Kuh08} for a survey. 

Kuhn's proof applies the Bousfield-Kuhn functor to a sequence of generalizations of the Kahn-Priddy splitting.
In this  note, we  use just the Bousfield-Kuhn functor applied to the
Kahn-Priddy splitting.
The key observation is that Tate vanishing for a finite group $G$ is equivalent to the localized
transfer map $BG_+ \to S^0$ admitting a section, and this can be proved directly when $G = C_p$
using the Kahn-Priddy theorem. 
Thus, we obtain a simplification of Kuhn's argument \cite[Sec.
3]{Kuh04a},
avoiding the use of results such as the $C_p$-Segal conjecture.

\subsection*{Acknowledgments} 
We would like to thank the referee for several helpful comments. 
The first author was supported by Lars
Hesselholt's Niels Bohr Professorship. The second author was supported by the NSF
Graduate Fellowship under grant DGE-114415.

\section{The proof}

\begin{lemma}\label{unit}
Let $R$ be a multiplicative cohomology theory and $K$ a pointed connected
CW-complex with basepoint $k$.  If $r\in R^0(K)$ restricts to a unit in $R^0(\{k\})$, then $r$ itself is a unit.
\end{lemma}
\begin{proof}
We prove that for every pointed subcomplex $K'\subset K$, the map $
R^\ast(K')\rightarrow R^\ast(K')$ given by multiplication by $r$ is an isomorphism.  This is true for $K'$ a
point by hypothesis. If $K' \subset K$ is a finite-dimensional subcomplex, an Atiyah-Hirzebruch spectral sequence argument shows
that the kernel of $R^0(K') \to R^0(\left\{k\right\})$ is nilpotent, which
forces $r$ to restrict to a unit in $R^0(K')$. Then a five lemma argument with the Milnor sequence implies it for every $K'$.  Taking $K'=K$ we conclude.
\end{proof}


Throughout, let $L\colon\sp\rightarrow \sp$ be a Bousfield localization
and $G$ a finite group.
\begin{definition} 
Given an object $X \in \fun(BG, L\sp)$, the ($L$-local) \emph{Tate construction}
$X^{tG} \in L\sp$ is the cofiber of the \emph{norm map} $X_{hG} \to X^{hG}$.
\end{definition} 
Note that the homotopy orbits $X_{hG}$ are computed in $L \sp$, so that they
are the $L$-localization of the homotopy orbits in $\sp$, and this is the
$L$-localization of the usual Tate construction in spectra.

Our basic observation is the following proposition. 

\begin{proposition}\label{equiv}
The following conditions are equivalent:

\begin{enumerate}
\item For every $G$-object $X$ of $L\sp$, the norm map $X_{hG}\rightarrow
X^{hG}$ in $L\sp$ is an equivalence, i.e., $ X^{tG} = 0$.
\item Condition 1 holds just for $X=LS$ with trivial $G$-action.
\item The transfer map 
\begin{equation}
\label{transfermap}
\Sigma^\infty_+ BG\rightarrow \Sigma^\infty_+ \ast \end{equation}
of spectra admits a section after applying $L$.
\end{enumerate}
\end{proposition}
\begin{proof}
Clearly 1 $\Rightarrow$ 2; for the converse, we 
use that $X^{tG}$ is a module over $LS^{tG}$ (cf. \cite[Prop.
I.3.5]{GreenleesMay}). 

We now prove that $2 $ and $ 3$ are equivalent. 
We use the basic diagram in $L \sp$ (compare \cite[Sec. I.5]{GreenleesMay})
\begin{equation}  \xymatrix{ (LS)_{hG} \ar[r]^N \ar[rd]^{f} &  (LS)^{hG} \ar[d]^{r}  \ar[r] &  (LS)^{tG}  \\
& LS 
}
\end{equation}
Here: 
\begin{enumerate}
\item  
The Tate construction $(LS)^{tG}$ is a ring spectrum and the map from
$(LS)^{hG}$ is a multiplicative 
map.
\item
 The map $f$ is the $L$-localization of the transfer \eqref{transfermap}. 
 \item 
The map $r$ identifies with the map $F( BG_+, LS) \to LS$ given by the 
basepoint of $BG$.
\item 
 The horizontal row is a cofiber sequence. 
\end{enumerate}

Suppose $2$ holds. Then $N$ is an equivalence. Since $r$ has a section, it
follows from the diagram that $f$ does as well, as desired. 

Finally, suppose $3$ holds, i.e., $f$ has a section. 
To show that $(LS)^{tG} = 0$,  the diagram shows that it  suffices to see that the
induced  map $N\colon \pi_0 ( (LS)_{hG}) \rightarrow \pi_0 ( (LS)^{hG})$ has image
containing a unit, which will then map to zero in $\pi_0 ( (LS)^{tG})$.  
Since $f$ has a section, it follows that 
 there exists $x \in \pi_0   ( (LS)_{hG})$ whose image in $\pi_0 ( LS)$
under $f = r \circ N$ is equal to $1$. 
It follows that $Nx \in \pi_0 ( (LS)^{hG}) = (LS)^0(BG)$ is a unit in view of 
Lemma~\ref{unit}, which completes the proof. 
\end{proof}

Next we need a reduction to the group $C_p$. 
\begin{lemma}\label{reduce}
Let $L$ be a Bousfield localization of spectra.  If the equivalent conditions
of Proposition~\ref{equiv} are satisfied for every group $G$ of prime order, then they are satisfied for every finite group $G$.
\end{lemma}

\begin{proof}
This follows from \cite[Lemmas~2.7 and 2.8]{Kuh04a}. 
\end{proof}

\begin{theorem}[Compare Kahn-Priddy \cite{KP}, Segal \cite{Seg74}]\label{kp}
The transfer map $\Sigma^\infty_+
BC_p\rightarrow \Sigma^\infty_+\ast$ admits a section after
applying the functor $\Omega^{\infty+1}$.
\end{theorem}

Note that the Kahn-Priddy theorem is usually stated for $\Sigma_p$ replacing
$C_p$. However, the result as stated follows because it is reduces to a
statement at the prime $p$ (in fact, $(\Sigma^\infty_+ B \Sigma_p)[1/p] \simeq
S^0[1/p]$) and the transfer exhibits $(\Sigma^\infty_+ B
\Sigma_p)_{(p)}$ as  a summand of $(\Sigma^\infty_+BC_p)_{(p)}$.  
Note also that the connected parts of the spectra in question are all torsion
and split into a product of their $q$-localizations for primes $q$. 

Putting things together, we thus obtain our main result. 

\begin{theorem}
Suppose $L$ is a Bousfield localization of spectra such that there
exists a functor $\Phi\colon \mathcal{S}_\ast\rightarrow L\sp$
such that $\Phi\Omega^\infty\simeq L$.  Then the equivalent conditions of
Proposition \ref{equiv} are satisfied for every finite group $G$.
In particular, Tate constructions in $L\sp$ vanish.
\end{theorem}
\begin{proof}
By Lemma~\ref{reduce}, it suffices to assume that 
$G$ has prime order.  Applying $\Phi$ to the section of Theorem~\ref{kp}, we
deduce that condition 3 holds for such $G$, concluding the proof. 
\end{proof}

Using the $K(n)$-local Bousfield-Kuhn functors of \cite{Kuh89} and their generalization to the
telescopic setting in \cite{Bou01}, we recover: 

\begin{corollary}[{Cf. \cite[Theorem 1.5]{Kuh04a}}]
The equivalent conditions of Proposition~\ref{equiv} hold for $E = T(n)$.
\end{corollary} 
\begin{remark} 
We also obtain as a consequence that there can be no analog of the
Bousfield-Kuhn functor for $E$-localization when $E = K(n_1) \vee K(n_2)$ for
$n_1 <
n_2$. In fact, we have that $L_{K(n_1)}(E_{n_2})^{tC_p} \neq 0$. 
\end{remark} 

\bibliographystyle{amsalpha}
\bibliography{Tate}

\end{document}